\documentclass[11pt]{article}
\usepackage{amssymb}
\usepackage{amsthm}
\usepackage{amsmath}
\usepackage{amscd}
\newtheorem{theorem}{Theorem}[section]
\newtheorem{corollary}[theorem]{Corollary}
\newtheorem{lemma}[theorem]{Lemma}
\newtheorem{proposition}[theorem]{Proposition}
\theoremstyle{definition}
\newtheorem{definition}[theorem]{Definition}
\newtheorem{remark}[theorem]{Remark}
\newtheorem{example}[theorem]{Example}

\begin{document}

\centerline{} \centerline{} \centerline{\Large{\bf {The Space of
Integrable Dirac Structures on Hilbert C*-Modules}}}
\centerline{}\centerline{\Large{\bf
 }}
     \centerline{\bf{ Vida Milani}}
  \centerline{ {\it Dept. of Math., Faculty of Math. Sciences, Shahid Beheshti university, Iran}}

\centerline{{\it School of Mathematics, Georgia Institute of Technology, Atlanta, USA}}

\centerline{e-mail: v-milani@cc.sbu.ac.ir}

\vline

  \centerline{\bf {Seyed M.H. Mansourbeigi}}
  \centerline{{\it Dept. of Electrical Engineering, Polytechnic University, NY, USA}}
  \centerline{e-mail: s.masourbeigi@ieee.org}
\vline

\centerline{\bf{Hassan Arianpoor}}
  \centerline{{\it Dept. of Math., Faculty of Math. Sciences, Shahid Beheshti university, Iran}}
  \centerline{e-mail: h\_arianpoor@sbu.ac.ir }
  \vline

\begin{abstract}

 In this paper we interpret the integrability of the Dirac structures
 on some Hilbert C*-modules in terms of an automorphism group.
  This is the group of orthogonal transformations
 on the Hilbert C*-module of sections of a Hermitian vector
 bundle over an smooth manifold $M$. Some topological properties
 of the group of integrable Dirac structures are studied. In some
 special cases it is shown that the integrability condition
 corresponds to the solutions of a partial differential equation.
 This is explained as a necessary and sufficient condition.

\end{abstract}


\emph{AMS subject class: 46B20, 46C05, 53B35}

\textit{}\\

\section{Introduction}

The idea of a Poisson bracket on the algebra of smooth functions
on a smooth manifold $M$ goes back to Dirac [4]. The underlying
structure of any Hamiltonian system is a Poisson algebra. Courant
and Weinstein [2] presented an approach to unify the geometry of
Hamiltonian vector fields and the Poisson brackets (unification of
Poisson and symplectic geometry). In both of these geometries, the
Poisson algebra is $C^{\infty}(M)$ and the bracket is given by a
specific bivector field on $M$ [4].

As a generalization of Poisson and presymplectic structures, the
theory of Dirac structures on vector spaces and their extension to
manifolds was considered by Courant and Weinstein [1,2]. These are
smooth subbundles of the direct sum bundle $TM \oplus T^*M$ of the
tangent and cotangent bundles, maximally isotropic under the
pairing

$$< (X, \omega) , (Y , \mu)>_+ =: \frac{\omega(Y) + \mu(X)}{2}$$

on $TM \oplus T^*M$.

Dorfman [5] developed the algebraic version of Dirac structures.
The generalization of Dirac structures on real and complex Hilbert
spaces and on Hermitian modules are considered in [7,8].

Our motivation in this paper has been the following consideration:

The integrability of Dirac structures on manifolds was introduced
by Courant [1]. It is important in that it leads to a Poisson
algebra of functions, making it possible to construct the
classical mechanics on the manifold [3,4,5].

The object of this paper is the interpretation of integrable Dirac
structures on pre-Hilbert C*-modules in such a way that we can
specify the moduli space of the group of integrable Dirac
structures on some Hilbert C*-modules.. The paper is organized in
the following manner.

First we give some preliminaries on the basic concepts of Dirac
structures on modules and on $TM$; the tangent bundle of the
smooth manifold $M$ and then introduce the notion of integrable
Dirac structures on modules.  After that we show in details how a
Dirac structure on $TM$ can be constructed out of an orthogonal
transformation of a Hilbert C*-module; the module of sections of a
Hermitian vector bundle on $M$. This enables us to define the
integrability of Dirac structures in terms of the orthogonal
transformations and go through their topological properties. A
necessary and sufficient condition for the integrability of a
Dirac structure is obtained as a solution to some certain partial
differential equation.

\section{Dirac Structures on pre-Hilbert C*-modules}

The concepts in this section are based on the references [6,8].

Let $\mathcal{A}$ be a C*-algebra, $H$ a right
$\mathcal{A}$-module. The action of an element $a \in \mathcal{A}$
on $H$ is denoted by $x.a$ for $x \in H$. $H$ together with a
sesquilinear form $<,>: H \times H \rightarrow \mathcal{A}$ with
the following properties

i)$<x,x> \geq 0 ; \forall x \in H$

ii)$<x,x> = 0$, implies $x = 0.$

iii)${<x,y>}^* = <y,x> ; \forall x,y \in H.$

iv)$<x,y.a> = <x,y>a ; \forall x,y \in H ; \forall a \in
\mathcal{A}.$

is called a {\it pre-Hilbert module}. For $x \in H$ let $\parallel
x\parallel_H =: \parallel <x,x> \parallel   ^{1/2}$. If the normed
space $(H,\parallel - \parallel)$ is complete, then $H$ is called
a {\it Hilbert C*-module}.

In this paper all the Hilbert C*-modules have the property that
for each nonzero $x \in H$, $2x \neq 0$.\\

\begin{example}

Let $M$ be a smooth compact n-manifold and $\pi : E \rightarrow M$
be a Hermitian vector bundle over $M$. Let $\mathcal{A}$ be the
C*-algebra of continuous functions on $M$. Let $H$ be the
$\mathcal{A}$-module of sections of this vector bundle. Then $H$
becomes a pre-Hilbert $\mathcal{A}$-module. In particular when
$\pi : TM \rightarrow M$ is the tangent bundle, the Hermitian
inner product enables us to identify this bundle with its dual
$T^*M$ and $\Gamma(TM)$; the $\mathcal{A}$-module of vector fields
on $M$ is identified with its dual $\Gamma(T^*M)$; the
$\mathcal{A}$-module of first order differential forms on $M$.\\

\end{example}

\begin{definition}

Let $H$ be a pre-Hilbert C*-module. Let $\tau : H \times H
\rightarrow H \times H$ be the flip operator defined by $\tau(x,y)
= (y,x)$ for $x,y \in H$. A submodule $L \subset H \times H$ is
called a {\it Dirac structure on $H$} if $L$ and $\tau (L)$ are
ortho-complementary.\\

\end{definition}

\section{Dirac Structures on Tangent Bundles}

\begin{definition}

[2] Let $M$ be a smooth n-manifold. A {\it Dirac structure on the
tangent bundle $TM$} is a maximally isotropic subbundle $L$ of the
Whitney sum bundle $TM \oplus T^*M$ under the pairing
$$<(X,\omega) , (Y,\mu)>_+ = \frac{1}{2}(\mu(X) + \omega(Y))$$

for $X,Y \in \Gamma(TM)$ and $\omega , \mu \in \Gamma(T^*M)$

\end{definition}

\begin{remark}

Let $\tau : TM \oplus T^*M \rightarrow T^*M \oplus TM $ be the
flip strong bundle isomorphism defined by $\tau (X,\omega) =
(\omega , X)$ for $X \in \Gamma(TM)$ and $\omega \in
\Gamma(T^*M)$. Furthermore let $X=(X_1,...,X_n)$ and
$\omega=(\omega_1,...,\omega_n)$ be respectively the local
coordinate functions of $X , \omega$ in a coordinate system on
$M$. The identification between the tangent and cotangent bundles
explained in example 2.1 shows that to each $X \in \Gamma(TM)$
there corresponds its dual $\omega_X \in \Gamma(T^*M)$ having the
same coordinates as $X$. Also to each $\omega \in \Gamma(T^*M)$
there corresponds an $X_\omega \in \Gamma(TM)$ having the same
coordinate functions as $\omega$.
\end{remark}

With these conventions we have

\begin{proposition}

The subbundle $L \subset TM \oplus T^*M$ is a Dirac structure on
$TM$ if and only if $L$ and $\tau(L)$ are ortho-complementary.

\end{proposition}

\begin{proof}

Suppose $L$ is a Dirac structure on $TM$. For
$(X=(X_i)_{i},\omega=(\omega_i)_i) \in L \cap \tau(L)$, we have
$((X_1,...,X_n),(\omega_1,...,\omega_n)),
((\omega_1,...,\omega_n), (X_1,...,X_n)) \in L$ and since $L$ is
isotropic, this implies that for each $i=1,...,n$,
$X_i=\omega_i=0$. So $L \cap \tau(L) = 0$. Also since $L$ is
maximally isotropic, $L$ and $\tau(L)$ are orthogonal and $L
\oplus \tau(L) = TM \oplus T^*M$.

Conversely if $L$ and $\tau(L)$ are ortho-complementary, then
obviously $L$ is maximally isotropic with respect to the pairing
$<,>_+$.

\end{proof}

\begin{remark}
Let $P_1,P_2$ be respectively the first and second projections on
$TM \oplus T^*M$. Let $L \subset TM \oplus T^*M$ be a Dirac
structure on $TM$. since $L$ is a Dirac structure, then for
$(X,\omega),(Y,\beta) \in L$,
$$<P_1(X,\omega),P_2(Y,\beta)> +
<P_2(X,\omega),P_1(Y,\beta)> = 0$$ In particular this is true for
the basis elements of $L$, so it implies $P_1P_2^* + P_2P_1^* =
0$.

\end{remark}

\begin{proposition}

If $L \subset TM \oplus T^*M$ is a Dirac structure on $TM$, then
the restriction of $P_1+P_2$ and $P_1-P_2$ to $L$ are strong
bundle isomorphisms.
\end{proposition}

\begin{proof}

For $X \in \Gamma(TM)$ and $\omega \in \Gamma(T^*M)$ with local
coordinates $X=(X_i)_i, \omega=(\omega_i)_i$, $i=1,...,n$, if
$(P_1+P_2)(X,\omega)=0$, $X_i=-\omega_i$ for each $i=1,...,n$. So
$(X,\omega) \in L \cap \tau(L)=0$, since $L$ is a Dirac structure.
Thus $P_1+P_2$ is injective. The same argument shows that
$P_1-P_2$ is injective.

Let us identify $TM$ with $T^*M$ via the Hermitian inner product.
let $\eta$ be the trivial Hermitian vector bundle of rank n over
$M$, $\mathcal{A}$ the C*-algebra of continuous functions on $M$
and $H$ the Hilbert $\mathcal{A}$-module of sections of $\eta$.
For each $f=(f_1,...,f_n) \in H$, let $X \in \Gamma(TM)$ and
$\omega \in \Gamma(T^*M)$ both have $f$ as their local coordinate
functions. Since $L$ and $\tau(L)$ are ortho-complementary, there
are $(Y,\beta) \in L$ and $(Z,\mu) \in \tau(L)$ with local
coordinates $Y=(Y_i)_i$, $Z=(Z_i)_i$, $\beta=(\beta_i)_i$ and
$\omega=(\omega_i)_i$, such that
$$(X,\omega) = (Y,\beta) \oplus (Z,\mu)$$
Thus $f_i=Y_i+Z_i=\beta_i+\mu_i$ and $Y_i-\mu_i=\beta_i-Z_i$ for
all i. So $((Y_i-\mu_i)_i , (Y_i-\mu_i)_i) = ((Y_i-\mu_i)_i ,
(\beta_i-Z_i)_i) = (Y_i,\beta_i)_i - (\mu_i,Z_i)_i \in L \cap
\tau(L) = 0$. Then $Y_i = \mu_i , \beta_i = Z_i$ for all i. And so
$X=(f_1,...,f_n)=(P_1+P_2)(Y,\beta)$, means that $P_1+P_2$ is
surjective. In the same way we can see that $P_1-P_2$ is
surjective.
\end{proof}

\begin{remark}

With the notations of the previous proposition, if $Aut(TM)$ be
the group of strong bundle automorphisms of the bundle $TM$, then
$A=(P_1+P_2)(P_1-P_2)^{-1} \in Aut(TM)$ (after the identification
of $TM$ with $T^*M$. Also by the restriction of $P_1 , P_2$ on the
sections, we can interpret $A \in Aut(\Gamma(TM))$.

\end{remark}

\begin{lemma}

With the notations of the previous remark, $A \in Aut(\Gamma(TM))$
is orthogonal.
\end{lemma}

\begin{proof}

$$AA^* = (P_1+P_2)(P_1-P_2)^{-1}(P_1^*-P_2^*)^{-1}(P_1^*+P_2^*)$$
$$=(P_1+P_2)((P_1^*-P_2^*)(P_1-P_2))^{-1}(P_1^*+P_2^*)$$
$$=(P_1+P_2)(P_1^*P_1-P_1^*P_2-P_2^*P_1+P_2^*P_2)^{-1}(P_1^*+P_2^*)$$
$$=(P_1+P_2)(P_1^*P_1+P_1^*P_2+P_2^*P_1+P_2^*P_2)^{-1}(P_1^*+P_2^*)$$
$$=(P_1+P_2)((P_1^*+P_2^*)(P_1+P_2))^{-1}(P_1^*+P_2^*) = I$$

Where we have used the fact in remark 3.4 that
$P_1^*P_2+P_2^*P_1=0$.

\end{proof}

\begin{proposition}

With the notations of the remark 3.2, let $B \in Aut(\Gamma(TM))$
be orthogonal, then
$$L_B = \{((I+B)X , (I-B)\omega_X) ; X \in \Gamma(TM)\}$$
is a Dirac structure on $TM$.

\end{proposition}

\begin{proof}

Let $((I+B)X,(I-B)\omega_X),((I+B)Y,(I-B)\omega_Y) \in L_B$. Then

from example 2.1 and remark 3.2, we have the following equations

$$<\omega_Y,BX> = <B\omega_X,Y> , <\omega_X,BY> = <B\omega_Y,X>$$
and also since $B$ is orthogonal,
$$<\omega_Y,X> = <B\omega_Y,BX> , <\omega_X,Y> = <B\omega_X,BY>$$

These equations imply
$$<(I-B)\omega_Y , (I+B)X> + ((I-B)\omega_X , (I+B)Y> = 0$$
and so $L_B$ is isotropic.

Now if $(Z,\alpha) \in \Gamma(TM) \oplus \Gamma(T^*M)$ be such
that $L_B \cup \{(Z, \alpha)\}$ is isotropic, then for each
$((I+B)X,(I-B)\omega_X) \in L_B$, we have
$$0 = <((I+B)X ,(I-B)\omega_X) , (Z , \alpha)>_+ $$
$$= <\alpha , (I+B)X> + <(I-B)\omega_X , Z>$$
$$=<\alpha , X> + <\alpha , BX> + <\omega_X , Z> - <B\omega_X ,
Z>$$
$$=<\alpha , X> + <\alpha , BX> - <\omega_Z , BX> + <\omega_Z,
X>$$
$$=<\alpha + \omega_Z , X> + <\alpha - \omega_Z , BX> $$
And so $<B(\alpha + \omega_Z) , BX> + <\alpha - \omega_Z , BX> =
0$. Thus $B(\alpha + \omega_Z) = \omega_Z - \alpha$. In the same
way $B(Z + Z_\alpha) = Z - Z_\alpha.$

where $Z_\alpha \in \Gamma(TM), \omega_Z \in \Gamma(T^*M)$ are
respectively the corresponded duals to $\alpha$ and $Z$ as in
remark 3.2.

So $Z = \frac{1}{2}(I+B)(Z + Z_\alpha)$ and $\alpha =
\frac{1}{2}(I-B)(\alpha + \omega_Z)$. Thus $(Z,\alpha) \in L_B$
and $L_B$ is maximal.

\end{proof}

\begin{proposition}

Any Dirac structure $L \subset TM \oplus T^*M$ on $TM$ is of the
form $L_B$ for some $B \in Aut(\Gamma(TM))$.

\end{proposition}

\begin{proof}

We  have seen that if $L$ is a Dirac structure on $TM$, then the
restrictions of $P_1+P_2, P_1-P_2$ to $L$ are isomorphisms and so
$A=(P_1+P_2)(P_1-P_2)^{-1} \in Aut(TM)$ is orthogonal. Now to this
$A$ there corresponds a $B \in Aut(\Gamma(TM))$ which is
orthogonal and so $L$ is the Dirac structure $L_B$ corresponded to
$B$.

\end{proof}

\section{The Topology of Integrable Dirac Structures}

\begin{definition}

[1] Let $L \subset TM \oplus T^*M$ be a Dirac structure on $TM$.
Then $L$ is said to be {\it integrable} if for each $(X , \omega
),(Y , \mu) \in L$, we have
$$([X,Y] , \{\omega , \mu\}) \in L$$
where $\{\omega , \mu\} = X(d \mu) - Y(d \omega) +
\frac{1}{2}d(X(\mu)-Y(\omega))$.

\end{definition}

When $L$ is a Dirac structure on $TM$, in proposition 3.9 we have
shown that, $L$ is of the form $L_B$ for some orthogonal $B \in
Aut(\Gamma(TM))$. We have the following definition

\begin{definition}

For orthogonal $B \in Aut(\Gamma(TM))$, the Dirac structure $L_B$
is {\it integrable} if for each pair
$((I+B)X,(I-B)\omega_X),((I+B)Y,(I-B)\omega_Y) \in L_B$, we have
$$(I+B)\{(I-B)\omega_X,(I-B)\omega_Y\} = (I-B)[(I+B)X,(I+B)Y]$$

$B \in Aut(\Gamma(TM))$ is called {\it integrable automorphism} if
$L_B$ is an integrable Dirac structure.

\end{definition}

By a straight forward calculation we can see
\begin{lemma}
The above two definitions for the integrability of Dirac
structures are equivalent.
\end{lemma}

\begin{corollary}
For nonzero real number $r$, $L_{\pm rI}$ are integrable only if
$r = \pm 1$.

\end{corollary}

\begin{proof}
Since for orthogonal $B \in Aut(\Gamma(TM))$ the eigenvalues of
$B$ are only $\pm 1$, so $\pm rI \in Aut(\Gamma(TM))$ for $r \neq
1$ are not integrable.
\end{proof}

Now let $\mathbb{R}$ be the field of real numbers, $\mathbb{R}^2$
the Euclidean space with the two coordinate functions $x,y$, $R$
the $\mathbb{R}$-ring of degree two polynomials in $x,y$ and $M=
\Gamma(T\mathbb{R}^2)$ the $R$-module of vector fields on
$\mathbb{R}^2$.
\\

\begin{proposition}

 $Aut(M)$ is in one to one correspondence with $GL_2(\mathbb{R})
\times \mathbb{R}^8$.

\end{proposition}

\begin{proof}

If $A = (a_{ij})_{i,j = 1,2} \in Aut (M)$ and $a_{ij} = a_{ij}^0 +
a_{ij}^1 x + a_{ij}^2 y$, then $det A$ is invertible. On the other
hand

$$det A = a_{11}^0 a_{22}^0 - a_{12}^0 a_{21}^0 + ....$$

is invertible iff $a_{11}^0 a_{22}^0 - a_{12}^0 a_{21}^0 = 1$.

So the map $$\theta : Aut(M) \rightarrow GL_2 (\mathbb{R}) \times
\mathbb{R}^8$$

defined by
$$\theta (a_{ij})_{i,j = 1,2} = ((a_{ij}^0)_{i,j=1,2,...} , ...)$$

is one to one and onto.
\end{proof}

With the notations of the previous proposition, a modification of
the definition of the Dirac structure $L_A$ for
 $A \in Aut(M)$ where $M=\Gamma(T\mathbb{R}^2)$
 is
as follows

\begin{definition}

If $A \in Aut(M)$, $\theta (A) = (A_0 , A_1)$, then
$$L_{A_0} = \{ (X + A_0 X , X - A_0 X) ; X \in M \}$$

is called {\it a Dirac structure on $M$}.
\end{definition}

For simplicity we denote $L_{A_0}$ by $L_A$.

From the proposition 4.5 it follows that each $A \in Aut(M)$ can
be considered as an element of $GL_2(\mathbb{R})$. With this
convention:

The set of all integrable automorphisms with the norm defined by
$$\parallel A \parallel _{\infty} = sup_{p \in \mathbb{R}^2} \{ \parallel A(p)\parallel
^2 + \sum_{i=1,2}\parallel \partial _i A(p)\parallel ^2
\}^{\frac{1}{2}}$$ for each integrable $A \in Aut(M)$, is a
topological group. This group is denoted by $I_D(M)$.
\\

\begin{proposition}
{\it If $A \in Aut(M)$ is integrable, and if $A \neq -I$, then
there exists a curve connecting $A$ to $I$.}
\end{proposition}

\begin{proof}

For $t \in [0,1]$, define $$f : [0,1] \rightarrow Aut(M)$$ by
$$f(t) = \frac{(1-t) + (1+t)A}{(1+t) +(1-t)A}$$

$f$ is continuous and $f(0) = I , f(1) = A$.

\end{proof}

\begin{proposition} $I_D(M)$ has the following properties,

i)$I_D(M)$ is Hausdorff.

ii)$I_D(M)$ is not connected.

iii)$I_D(M)$ is closed in $O(2)$.

iv)$I_D(M)$ is not open.

\end{proposition}

\begin{proof}

i,ii)$A \in Aut(M)$ is orthogonal, so $I_D(M) \subset O(2)$ and so
it is Hausdorff. From corollary 4.4, it follows that $I_D(M)$ has
two components, one contains $I$ and the other contains $-I$.

iii)The derivative map is continuous, from the definition of the
norm on $I_D(M)$, we see that $I_D(M)$ is closed.

iv)$-I$ is integrable and is the isolated point of $I_D(M)$, so
$I_D(M)$ is not open.

\end{proof}

Set $\partial_1 = \frac{\partial}{\partial x}$ and $\partial_2 =
\frac{\partial}{\partial y}$.

\begin{proposition}

{\it A necessary and sufficient condition for $A =
(a_{uv})_{u,v=1,2} \in Aut(M)$ to be integrable is that for $m,i,k
= 1,2$, $A$ satisfies the following differential equation}

$$\partial_i(a_{mk}) - \partial_k(a_{mi}) + \sum _{l=1,2}(a_{li}
\partial_l(a_{mk}) - a_{lk} \partial_l(a_{mi})) + \sum
_{j=1,2}a_{ji} \partial_m(a_{jk}) + \sum _{l,j=1,2}a_{mj} a_{lj}
\partial_j(a_{lk}) = 0.$$

\end{proposition}

\begin{proof}

Set $(I-A)dx_i = \alpha$ and $(I-A)dx_k = \beta$. Then using the
notations of the Remark 3.2,

$$\alpha = -\sum_{j=1,2} a_{ji} dx_j + dx_i$$
$$\beta = -\sum_{l=1,2} a_{lk} dx_l + dx_k$$

So $$X_{\alpha} = -\sum_{j=1,2}a_{ji} \partial_j + \partial_i$$
$$X_{\beta}= -\sum_{l=1,2}a_{lk} \partial_l + \partial_k$$

So $A$ is integrable iff for $i,k = 1,2$,

$$(I+A)\{ \alpha , \beta \} = (I-A)[X_{\alpha} , X_{\beta}]$$

On the other hand

$$X_{\alpha}( \beta ) = \beta (X_{\alpha}) = \sum_{j=1,2}a_{jk}
a_{ji} - a_{ik} - a_{ki} + \delta_{ki}$$

$$X_{\beta}( \alpha ) = \alpha (X_{\beta}) = \sum_{l=1,2}a_{li}
a_{lk} - a_{ki} - a_{ik} + \delta_{ik}$$

Also we can write
$$d \alpha = - \sum_{t=1,2}\sum_{j=1,2}\partial_t (a_{ji}) dx_t
dx_j$$

$$d \beta = - \sum_{t=1,2}\sum_{l=1,2}\partial_t (a_{lk}) dx_t
dx_l$$

So

$$X_{\alpha}(d \beta) = d \beta (X_{\alpha}) =$$
$$\sum_{j,l=1,2}\partial_j (a_{lk}) a_{ji} dx_l -
\sum_{t,j=1,2}\partial_t (a_{jk}) a_{ji} dx_t -
\sum_{l=1,2}\partial_i (a_{lk})dx_l +
\sum_{t=1,2}\partial_t(a_{ik})dx_t$$\\

in the same way

$$X_{\beta}(d \alpha) = d \alpha (X_{\beta}) $$
$$\sum_{j,l=1,2}\partial_l (a_{jk}) a_{lk} dx_j -
\sum_{t,l=1,2}\partial_t (a_{li}) a_{lk} dx_t -
\sum_{j=1,2}\partial_k (a_{ji})dx_j +
\sum_{t=1,2}\partial_k(a_{ti})dx_t$$

So
$$\{ \alpha , \beta \} = X_{\alpha}(d \beta) - X_{\beta}(d \alpha)
+ \frac{1}{2} d(X_{\alpha}( \beta ) - X_{\beta}( \alpha ))$$

$$= \sum_{l=1,2}(\sum_{j=1,2} \partial_j (a_{lk}) a_{ji} -
\sum_{j=1,2} \partial_j (a_{li}) a_{jk} dx_l)
-\sum_{l=1,2}(\sum_{j=1,2} \partial_l (a_{jk}) a_{ji} $$
$$-\sum_{j=1,2} \partial_l (a_{ji}) a_{jk} dx_l) -\sum_{l=1,2}
(\partial_i (a_{lk}) - \partial_k (a_{li})
dx_l)~~~~~~~~~~~~~~~~~~~~~~~~\dag$$

On the other hand

$$X_{\alpha}(X_{\beta}) = \sum_{l,j=1,2}a_{ji}
\partial_j(a_{lk})\partial_l -
\sum_{l=1,2}\partial_i(a_{lk})\partial_l$$

$$X_{\beta}(X_{\alpha}) = \sum_{l,j=1,2}a_{lk}
\partial_l(a_{ji})\partial_j -
\sum_{j=1,2}\partial_k(a_{ji})\partial_j$$

So

$$[X_{\alpha} , X_{\beta}] = X_{\alpha}(X_{\beta}) -
X_{\beta}(X_{\alpha})$$

$$= \sum_{l,j=1,2}(a_{ji}\partial_j(a_{lk}) -
a_{jk}\partial_j(a_{li}))\partial_l) -
\sum_{l=1,2}(\partial_i(a_{lk}) -
\partial_k(a_{li}))\partial_l~~~~~~~~~~~~~~~~~~~~~~~~\ddag$$

Now if we multiply both sides of the equation $\dag$ by $(I+A)$
and the equation $\ddag$ by $(I-A)$, and then setting them equal
to each other we obtain the following differential equation as the
necessary and sufficient condition for the integrability of $A$

$$\partial_i(a_{mk}) - \partial_k(a_{mi}) + \sum _{l=1,2}(a_{li}
\partial_l(a_{mk}) - a_{lk} \partial_l(a_{mi})) + \sum
_{j=1,2}a_{ji} \partial_m(a_{jk}) + \sum _{l,j=1,2}a_{mj} a_{lj}
\partial_j(a_{lk}) = 0.$$
\end{proof}

\begin{remark}

In the case of the n-dimensional space $\mathbb{R}^n$, we obtain
the following differential equation for $m,i,k=1,2,...,n$

$$\partial_i(a_{mk}) - \partial_k(a_{mi}) + \sum _{l=1,2}(a_{li}
\partial_l(a_{mk}) - a_{lk} \partial_l(a_{mi})) + \sum
_{j=1,2}a_{ji} \partial_m(a_{jk}) + \sum _{l,j=1,2}a_{mj} a_{lj}
\partial_j(a_{lk}) = 0.$$

\end{remark}

 \end{document}